\documentclass[11pt,a4paper]{amsart}
\usepackage{tikz}
\usepackage{tikz-cd}
\usepackage{geometry}
\usetikzlibrary{decorations.markings,decorations.pathmorphing,cd}
\usetikzlibrary{arrows}
\usetikzlibrary{calc}
\usepackage{pgfplots}
\pgfplotsset{every axis/.append style={
		axis x line=middle,    
		axis y line=middle,    
		axis line style={-,color=blue}, 
		xlabel={$x$},          
		ylabel={$y$},          
}}

\usepackage{amsmath}
\usepackage{amsthm}
\usepackage{amsfonts}
\usepackage{amssymb}
\usepackage{amsrefs}
\usepackage{enumerate}
\usepackage{enumitem}
\usepackage{graphicx}
\usepackage{hyperref}
\usepackage{nicefrac}
\usepackage{cancel}

\newcommand\enet[1]{\renewcommand\theenumi{#1} 
	\renewcommand\labelenumi{\theenumi}}

\DeclareMathOperator{\orb}{orb}

\DeclareMathOperator{\SF}{\Gamma_F}
\DeclareMathOperator{\SB}{\Gamma_S}

\def\cF{\mathcal{F}}

\def\A{\mathbb{A}}
\def\ZZ{\mathbb{Z}}

\def\CC{\mathbb{C}}
\def\DD{\mathbb{D}}

\def\FF{\mathbb{F}}

\def\PP{\mathbb{P}}

\def\rightmap#1{\smash{\mathop{\rightarrow}\limits^{#1}}}

\newcommand{\Conv}{\mathop{\scalebox{1.5}{\raisebox{-0.2ex}{$\ast$}}}}%

\newtheorem{thm}{Theorem}[section]  
\newtheorem{main-thm}{Theorem}  
\newtheorem{prop}{Proposition}[section]%
\newtheorem{proper}{Condition}[section]%
\newtheorem{main-conj}{Conjecture}[section]%
\newtheorem{cor}{Corollary}[section]
\newtheorem{lemma}{Lemma}[section]

\theoremstyle{remark}
\newtheorem{rem}{Remark}[section]

\theoremstyle{definition}
\newtheorem{dfn}{Definition}[section]
\newtheorem{exam}{Example}[section]

\makeatletter
\let\c@lemma\c@thm
\let\c@prop\c@thm
\let\c@propdef\c@thm
\let\c@proper\c@thm
\let\c@problem\c@thm
\let\c@conj\c@thm
\let\c@cor\c@thm
\let\c@rem\c@thm
\let\c@dfn\c@thm
\let\c@notation\c@thm
\let\c@exam\c@thm
\makeatother

\title
[The fundamental group of the complement of a generic fiber-type curve]
{The fundamental group of the complement of a generic fiber-type curve}

\author[Jos\'e I. Cogolludo-Agust{\'i}n and Eva Elduque]{J.I.~Cogolludo-Agust{\'i}n and Eva Elduque}
\address{Departamento de Matem\'aticas, IUMA\\
	Universidad de Zaragoza \\
	C.~Pedro Cerbuna 12 \\
	50009 Zaragoza, Spain.} 
\email{jicogo@unizar.es} 

\address{Departamento de Matem\'aticas, ICMAT\\ 
	Universidad Aut\'onoma de Madrid \\
	28049 Madrid, Spain.}
\email{eva.elduque@uam.es}

\makeindex

\begin{document}
	
	\thanks{The authors are partially supported by PID2020-114750GB-C31, funded by 
		MCIN/AEI/10.13039/501100011033. The first author is also partially funded by the Departamento de Ciencia, 
		Universidad y Sociedad del Conocimiento of the Gobierno de Arag\'on 
		(Grupo de referencia E22\_20R ``\'Algebra y Geometr\'{\i}a''). 
		The second author is partially supported by the Ram\'on y Cajal Grant RYC2021-031526-I funded by MCIN/AEI/10.13039/501100011033 and by the European Union NextGenerationEU/PRTR
	}
	
	
	\maketitle
	\begin{abstract}
	In this paper we describe and characterize the fundamental group of the complement to generic fiber-type curves,
	i.e. unions of (the closure of) finitely many generic fibers of a component-free pencil
	$F=[f:g]:\CC\PP^2\dashrightarrow \CC\PP^1$. We also observe that these groups have already appeared in the literature
	in examples by Eyral and Oka associated with curves of fiber-type. As a byproduct of our techniques we obtain
	a criterion to address Zariski's Problem on the commutativity of the fundamental group of the complement of
	a plane curve.
	\end{abstract}

	\section{Introduction}
	The present work is a continuation of a series of papers by the authors devoted to a particular class of divisors
	on surfaces called fiber-type curves. Fiber-type curves have been considered classically combining both topological
	and geometric perspectives
	(see~\cite{Shimada-fundamental,Arapura-fundamentalgroups,Arapura-geometry,ACM-characteristic,Bauer-irrational,Catanese-Fibred,Green-Lazarsfled-higher,Gromov-fundamental,Jost-Yau-Harmonic,Hillman-Complex}).
	In particular, we will focus on fiber-type curves on the complex projective
	plane $\PP^2$ that are a finite union of generic fibers. A precise definition is given as follows.

	\begin{dfn}[Plane fiber type curve in $\PP^2$]
	Let $C$ be a curve in the complex projective plane $\PP^2$. We say that a $C$ is a fiber-type
	curve if there exists a component-free pencil $F:\PP^2\dashrightarrow \PP^1$
	such that $C$ is the closure in $\PP^2$ of a finite number of fibers of~$F$.
	\end{dfn}

\begin{rem}\label{rem:Stein}
Note that, after Stein factorization, every fiber-type curve is a fiber-type curve associated with a
pencil $F : \PP^2\dashrightarrow \PP^1$ which has connected generic fibers (i.e., the
generic fibers of $F : \PP^2\setminus \mathcal B\rightarrow \PP^1$ are connected, where $\mathcal B$
is the finite set of base points of the pencil~$F$).
\end{rem}

	Even though any plane curve can be considered as a fiber-type curve by taking any line in the projective space
	of curves of a given degree, not every curve can be a union of generic members of a component-free pencil.
	We will refer to such curves as \emph{generic fiber-type curves} (see~\ref{sec:admissible}).
	There are strong topological restrictions
	in terms of the fundamental group of the complement of a generic fiber-type curve. A key role is played by the
	existence of multiple fibers in this pencil.

\begin{rem}\label{rem:2mult}
By \cite[Proposition 2.8]{ji-Libgober-mw}, the number of multiple fibers of a pencil
$F:\PP^2\dashrightarrow\PP^1$ with connected generic fibers is at most $2$. Hence, after a change
of coordinates in $\PP^1$, we may always assume that the multiple fibers of any pencil with connected
generic fibers lie over a subset of $\{[0:1],[1:0]\}$, and that the multiplicity of the fiber over
$[1:0]$ is greater or equal than the multiplicity of the fiber over~$[0:1]$.
\end{rem}

In light of Remarks~\ref{rem:Stein} and~\ref{rem:2mult}, every plane fiber-type curve is a fiber-type
curve associated to a pencil $F : \PP^2\dashrightarrow \PP^1$ satisfying the following condition:
\begin{proper}\label{cond}
$F : \PP^2\dashrightarrow \PP^1$ is of the form $F=[f_{kp}^q:f_{kq}^p]$ for some $p,q,k\in\ZZ_{>0}$
with $p\geq q$, where
\begin{itemize}
\item $f_{kp}$ and $f_{kq}$ are coprime homogeneous polynomials in $\CC[x,y,z]$ of degrees $kp$ and
$kq$ respectively.
\item
Neither $f_{kp}$ nor $f_{kq}$ is an $l$-th power of another polynomial in $\CC[x,y,z]$ for any $l\geq 2$.
\item
$F$ has connected generic fibers, so $\gcd(p,q)=1$.
\item
$F$ has no multiple fibers outside $\{[0:1],[1:0]\}\subset\PP^1$
\end{itemize}
\end{proper}

	
	Given a pencil $F$ and a point $P\in\PP^1$, we denote $C_P:=\overline{F^{-1}(P)}\subset \PP^2,$
	and, for any finite set $B\subset \PP^1$, we use $C_B:=\bigcup_{P\in B} C_P$ to denote the fiber-type curve
	corresponding to $F$ and~$B$.

	The original motivation for this paper comes from \cite{Oka-genericRjoin}, where Oka defines the following family
	of groups which will appear as fundamental groups of complements of some particular cases of generic fiber-type curves
	(see also Eyral-Oka~\cite{Eyral-Oka-RjoinType, Eyral-Oka-RjoinTypeII} for more examples where the genericity condition
	is relaxed).

\begin{dfn}\label{def:Gpqr}
	Let $p$, $q$ and $r$ be positive integers. \emph{Oka's group} $G(p;q;r)$ is given by the following presentation:
	\begin{equation}\label{eq:presGpqr}
		G(p;q;r)=
		\langle \omega, a_i (i\in\ZZ): \omega=a_{p-1}\cdot a_{p-2}\cdot\ldots \cdot a_1\cdot a_0, \omega^r=e, R_1, R_2\rangle,
	\end{equation}
	where $R_1$ are the (periodicity) relations of the form $a_i=a_{i+q}$ for all $i\in\ZZ$, and $R_2$ are the (conjugacy) relations of the form $a_{i+p}=\omega a_i\omega^{-1}$ for all $i\in\ZZ$.
\end{dfn}



\begin{exam}[Cor. 1.4 in \cite{Oka-genericRjoin}]\label{ex:Oka}
	Let $r\geq 0$. Let $G$ be the fundamental group of the complement of
	a fiber-type curve $C=\cup_{k=0}^{r} C_{[1:\omega^k]}$ associated to the pencil $F=[f:g]$, where
	\begin{itemize}
		\item $\omega$ is a primitive $(r+1)$-th root of unity,
		\item $d=\deg(f)=\deg(g)$,
		\item $f=\prod_{j=1}^{l_1} (x-\beta_jz)^{m_{1j}}$ and $g=\prod_{i=1}^{l_2} (y-\alpha_iz)^{m_{2i}}$,
		where $\alpha_1,\ldots,\alpha_{l_2},\beta_1,\ldots,\beta_{l_1}$ are mutually distinct complex numbers,
		\item $m_1:=\gcd(m_{1j})$ and $m_2:=\gcd(m_{2i})$ are coprime, and
		\item the curve $C$ is a union of generic members of the pencil $F$.
	\end{itemize}
	Under those conditions, Oka proves that $G\cong G\left((r+1)m_1;(r+1)m_2;\frac{d}{m_1}\right)$.
%
\end{exam}

%

Also in~\cite{Oka-genericRjoin}, Oka shows that $G(p;q;r)$ is a central extension of a free product of cyclic groups
by a finite cyclic group (see Proposition~\ref{prop:center}).
In fact, some finitely generated free products of cyclic groups appear as examples of Oka's groups, namely,
$G(sp;sq;q)\cong \FF_{s-1}*\ZZ_p*\ZZ_q$ for all $p,q,s\in \ZZ_{>0}$ such that $\gcd(p,q)=1$.
Oka's Example~\ref{ex:Oka} generalizes his own celebrated work~\cite[\S8]{Oka-some-plane-curves}
(also see Dimca~\cite[Prop. \S4(4.16)]{Dimca-singularities} and Némethi~\cite{Nemethi-fundamental}).
Examples~\ref{ex:Oka} and~\ref{ex:Okapq} are both from the 70's.

\begin{exam}[{\cite[\S8]{Oka-some-plane-curves}}]\label{ex:Okapq}
Let $p,q>1$ be coprime integers, and let $C_{p,q}$ be the irreducible curve in $\PP^2$ defined as
\begin{equation}
\label{eq:cpq}
C_{p,q}:=V\left((x^p+y^p)^q+(y^q+z^q)^p=0\right),
\end{equation}
which is the closure of the fiber over $[1:-1]$ of the pencil $F=[(x^p+y^p)^q:(y^q+z^q)^p]$.
Then, $\pi_1(\PP^2\setminus C_{p,q})\cong\ZZ_p*\ZZ_q$.
\end{exam}

Other results about curves whose fundamental group of their complement is a free product can be found in~\cite{Oka-topology}.

Our first contribution in this direction is given in Proposition~\ref{prop:GpqrImpliesFiber}, where we show that
any curve whose fundamental group of its complement is one of Oka's groups must be of fiber-type, one
irreducible component per fiber. However, the pencil $F=[f:g]$ used might not be such that $f$ and $g$ are completely
reducible fibers (products of powers of degree-one polynomials).

The main result of this paper computes the fundamental group of any generic fiber-type curve in $\PP^2$.
In particular, we obtain that such fundamental groups are Oka's groups from Definition~\ref{def:Gpqr}.

\begin{thm}[Main Theorem]\label{thm:newgeneric}
	Let $F : \PP^2\dashrightarrow \PP^1$ be as in Condition~\ref{cond}. Let $B\subset \PP^1\setminus B_F$
	be a finite non-empty set, and let $s:=\#B$. Then,
	$$\pi_1(\PP^2\setminus C_B)\cong G\left(sp;sq;kq\right).$$
	Moreover, assume that $V(f_{kp})$ (resp. $V(f_{kq})$) is irreducible and $\pi_1(\PP^2\setminus V(f_{kp}))\cong \ZZ_{kp}$ (resp. $\pi_1(\PP^2\setminus V(f_{kq}))\cong \ZZ_{kq}$). Then 
	$$\pi_1\left(\PP^2\setminus (C_B\cup V(f_{kp}))\right)\cong G\left((s+1)p;(s+1); k\right)$$
	(resp.
$
\pi_1\left(\PP^2\setminus (C_B\cup V(f_{kq}))\right)\cong  G\left((s+1)q;s+1;k\right)
$).
\end{thm}


It is worth mentioning that the proof of this result is more conceptual and avoids the braid monodromy and Zariski-Van
Kampen calculations needed in~\cite{Oka-genericRjoin, Eyral-Oka-RjoinTypeII}, which gives a new insight into these results.

As for the \emph{moreover} hypothesis $G=\pi_1(\PP^2\setminus V(f_{kp}))\cong \ZZ_{kp}$ in Theorem~\ref{thm:newgeneric},
this is equivalent to $V(f_{kp})$ being irreducible and $G$ abelian (see Lemma~\ref{lem:coprime}).
Finding a geometric criterion for the fundamental group of a curve complement to be abelian is known as
Zariski's Problem (see~\cite[\S2]{Libgober-complements} for a more detailed approach to this problem).
For instance smooth and nodal curves have complements with an abelian fundamental group
(see~\cite{Zariski-problem,Cheniot,Deligne-groupe}). Other criteria include the existence of high order flexes
(see~\cite{Zariski-problem}),
the existence of points of high multiplicity (see~\cite[Cor. \S4(3.8)]{Dimca-singularities}),
the positive self intersection of a resolution of the curve
(see~\cite{Nori-Zariski,Libgober-complements,Abhyankar-Zariski-algebraic}),
other non-negativity conditions (see~\cite{Orevkov-fundamental,Orevkov-commutant}),
or based on $\mu$-constant determinacy numbers (see~\cite[Cor. \S4(3.11)]{Dimca-singularities}).

As a consequence of the Main Theorem, we also obtain a contribution to Zariski's Problem.

\begin{cor}\label{cor:1or0multipleAbelianintro}
	If an irreducible curve $C$ is the closure of a generic fiber of a pencil $F:\PP^2\dashrightarrow \PP^1$ of
	degree $d$ with at most one multiple fiber, then $\pi_1(\PP^2\setminus C)\cong\ZZ_d$.
\end{cor}

Using this result one can produce infinitely many curves with complicated singularities, whose fundamental
groups of their complements are abelian and yet are not detected by the previously
known criteria for Zariski's Problem.

The paper is structured as follows. In Section~\ref{s:preliminaries} we give some preliminaries regarding fundamental
groups of curve complements and orbifold fundamental groups in a more general setting. In particular, in
Section~\ref{sec:fundamentalgroupsfiber}, we let $F:U\to S$ be a morphism between a smooth quasi-projective
variety and a projective curve, and consider fiber-type hypersurface complements of the form $U\setminus F^{-1}(B)$
for a finite set $B\subset S$.
This directly generalizes the current setting: If $F:U=\PP^2\setminus\mathcal B\to \PP^1$ is a non-constant pencil defined in
its maximal domain of definition, then $U\setminus F^{-1}(B)=\PP^2\setminus C_B$. In Section~\ref{sec:fundamentalgroupsfiber},
we develop group-theoretical methods to describe fundamental groups of generic fiber-type hypersurface
complements in the general setting of quasi-projective varieties such as Addition and Deletion Lemmas.

Section~\ref{s:Okagroups} studies Oka's groups from Definition~\ref{def:Gpqr}.
In Theorem~\ref{thm:OkaUnique}, we give a simpler finite presentation for Oka's groups than the one given
in Definition~\ref{def:Gpqr} and use it to give a characterization of such groups
which will allow us to identify them in different contexts and solve the isomorphism problem for the infinite
family of Oka's groups in Corollary~\ref{cor:isoGpqr}.

Finally, the Main Theorem~\ref{thm:newgeneric} is proved in Section~\ref{s:main}, describing
$\pi_1(\PP^2\setminus C)\cong G(p;q;r)$ for any generic fiber-type curve $C$ for certain non-zero integers $p,q,r$,
in particular, $\pi_1(\PP^2\setminus C)$ is a central extension of a free product of cyclic groups.
As a consequence of this, we obtain Corollary~\ref{cor:1or0multipleAbelianintro}.
The famous example of the equisingular stratum of sextics with an $\A_{17}$ singularity is discussed in
Example~\ref{exam:a17} using the techniques developed in this paper. We also improve Dimca's criterion for abelian
fundamental groups in the context of Zariski's Problem using $\mu$-constant determinacy in Example~\ref{exam:muconstant}.
Finally, we observe in Theorem~\ref{thm:additionPlane} that our Main Theorem~\ref{thm:newgeneric} can be extended
under some circumstances to non-generic fiber-type curves.

\subsection*{Acknowledgments}
The authors would like to thank Moisés Herradón Cueto for useful conversations and Andrei Jaikin-Zapirain for his help
with simplifying the proof of Theorem~\ref{thm:OkaUnique}.

\section{Preliminaries}\label{s:preliminaries}

\subsection{Admissible maps}\label{sec:admissible}
In this section, we talk about a generalized notion of the pencils satisfying Condition~\ref{cond}, and set notation
for the rest of the paper.

Following Arapura~\cite{Arapura-geometry}, we call a surjective algebraic morphism $F:U\to S$ from a smooth connected 
quasi-projective variety $U$ to a smooth connected quasi-projective curve $S$ \textit{admissible} if it has connected
generic fibers.

The minimal set of values $B_F$ for which 
$F:U\setminus F^{-1}(B_F)\to S\setminus B_F$ is a locally trivial fibration is finite~\cite{Thom-ensembles}. 
The points in $B_F$ are called \emph{atypical} values of $F:U\to S$. 
A fiber $F^{-1}(P)$ is called \emph{generic} if $P\notin B_F$.
We will distinguish between $F^{*}(P)$ as the pulled-back divisor and $C_P$ as 
its reduced structure. Using this notation, one can describe the set of multiple fibers as
\begin{equation}
\label{eq:multiple}
M_F=\{P\in S\mid F^{*}(P)=mD, m>1, \textrm{ for some effective divisor } D\}\subset B_F.
\end{equation}
Note that in general, the effective divisor $D$ in~\eqref{eq:multiple} need not be reduced. If $P\in S$, the 
\emph{multiplicity} 
of $F^{*}(P)$ is defined as $m\geq 1$ if $F^{*}(P)=mD$ for some $D$ and whenever $F^{*}(P)=m'D'$, then~$m'\leq m$.

If $F:\PP^2\dashrightarrow \PP^1$ is a component-free pencil, the definition of $B_F$, $M_F$, and generic fiber is
done using the induced map $F:\PP^2\setminus \mathcal B\to \PP^1$, where $\mathcal B$ is the finite set of base
points of the pencil.

\subsection{Fundamental groups of curve complements}\label{sec:fundamentalgroups}

\begin{dfn}[Meridian]
Let $X$ be a smooth quasi-projective variety and let $C=\cup_{i\in I} C_i$ be a curve in $X$, where $C_i$ are its 
irreducible components. Let $p\in X\setminus C$ be a base point. A positively (resp. negatively) oriented meridian around $C_i$ based at $p$ is the class in $\pi_1(X\setminus C,p)$ of a loop $\delta_i \star \hat\gamma_i \star \bar \delta_i$, where $\delta_i$ and $\hat\gamma_i$ are as follows:
\begin{itemize}
\item $p_i$ is a regular point on $C_i$,
\item $\DD_i\subset X$ is a closed disk centered at $p_i$ which is transversal to $C_i$ at $p_i$ such that $\DD_i\cap C=\{p_i\}$,
\item $\hat p_i$ is a point in the boundary of $\DD_i$,
\item $\hat\gamma_i$ is a loop based at $\hat p_i$ around the boundary of $\partial \DD_i$ travelled in the positive (resp. negative) orientation, and
\item $\delta_i:[0,1]\to X\setminus C$ is a path in $X\setminus C$ starting at the base point $\delta_i(0)=p$ and ending at $\delta_i(1)=\hat p_i$.
\end{itemize}
\end{dfn}

The following three results are well known.

\begin{lemma}
\label{lemma:meridians-conjugated}
Let $\gamma$ be a positively oriented meridian around $C_i$ based at $p$. Then, $\gamma'\in\pi_1(X\setminus C,p)$ is a positively oriented meridian around $C_i$ if and 
only if $\gamma'$ is in the conjugacy class of $\gamma$ in $\pi_1(X\setminus C,p)$.
\end{lemma}

\begin{proof}
See \cite{Nori-Zariski}, also~\cite[Prop. 1.34]{ji-pau} for a proof.
\end{proof}

\begin{lemma}
\label{lemma:meridians-generators}
Consider $X_i:=X\setminus (\cup_{j\in I\setminus\{i\}}C_j)$ and the map $(j_i)_*:\pi_1(X\setminus C)\to \pi_1(X_i)$ induced by the 
inclusion $X\setminus C\hookrightarrow X_i$. Then $(j_i)_*$ is surjective, and $\ker (j_i)_*$ is the normal closure of 
any meridian $\gamma_i$ in~$\pi_1(X\setminus C)$ around $C_i$. 
In particular, if $X$ is simply connected, then the set $\{\gamma_i\}_{i\in I}$ normally generates~$\pi_1(X\setminus C)$.
\end{lemma}

\begin{proof}
See~\cite{Nori-Zariski}. Also, as a consequence of~\cite[Lemma 2.3]{Shimada-Zariski}.
\end{proof}

\begin{lemma}[{\cite[\S4 Prop. (1.3)]{Dimca-singularities}}]\label{lem:coprime}
Suppose that $X=\PP^2$, and let $C=\cup_{i=1}^r C_i$ be a curve with $r$ irreducible components $C_1,\ldots,C_r$. Then,
$H_1(X\setminus C;\ZZ)\cong \ZZ^{r-1}\times \ZZ_d,$
where $d$ is the greatest common divisor of the degrees of the irreducible polynomials defining each of the components of $C$.
\end{lemma}

\subsection{Orbifold fundamental groups}
\label{sec:orbifold}

Let $S$ be a smooth projective curve of genus $g$, and choose a labeling map 
$\varphi:S\to\ZZ_{\geq 0}$ such that $\varphi(P)\neq 1$ only for a finite number of points, say 
$\Sigma=\Sigma_0\cup\Sigma_+\subset S$ for which $\varphi(P)= 0$ if $P\in \Sigma_0$ and 
$\varphi(Q)=m_{Q}>1$ if $Q\in \Sigma_+$. In this context, we will refer to this as an \emph{orbifold structure on} $S$.
This structure will be denoted by $S_{(s,\bar m)}$, where $s=\#\Sigma_0$, and $\bar{m}$ is a $(\#\Sigma_+)$-tuple 
whose entries are the corresponding $m_Q$'s. 

\begin{dfn}[Orbifold fundamental group]
The orbifold fundamental group associated with $S_{(s,\bar m)}$
and denoted by $\pi_1^{\orb}(S_{(s,\bar m)})$ is the quotient of
$
\pi_1(S\setminus \Sigma)$ by the normal closure of the subgroup $
\langle \mu_P^{\varphi(P)}, P\in \Sigma\rangle,
$
where $\mu_P$ is a meridian in $S\setminus \Sigma$ around $P\in \Sigma$.
Note that $\pi_1^{\orb}(S_{(s,\bar m)})$ is hence generated by
$
\{a_i,b_i\}_{i=1,\dots,g} \cup \{\mu_P\}_{P\in \Sigma}
$
and presented by the relations
\begin{equation}
\label{eq:rels}
\mu_P^{m_P}=1, \quad \textrm{ for } P\in \Sigma_+, \quad \textrm{ and } \quad 
\prod_{P\in \Sigma}\mu_P=\prod_{i=1,\dots,g}[a_i,b_i].
\end{equation}
\end{dfn}
In the particular case where $\Sigma_0\neq \emptyset$ ($s\geq 1$), \eqref{eq:rels} shows that 
$\pi_1^{\textrm{orb}}(S_{(s,\bar m)})$ is a free product of cyclic groups:
$$
\pi_1^{\textrm{orb}}(S_{(s,\bar m)})\cong \pi_1(S\setminus\Sigma_0)\Conv\left(
\Conv_{P\in\Sigma_+}\left(\frac{\ZZ}{m_P\ZZ}\right)\right)\cong
\FF_{r}*\ZZ_{m_1}*\dots*\ZZ_{m_n},
$$
where $r=2g-1+\# \Sigma_0=2g+s-1$, $n=\# \Sigma_+$, and $\bar m=(m_1,\dots,m_n)$.

\begin{dfn}
Let $X$ be a smooth algebraic variety. A dominant algebraic morphism $F:X\to S$ defines an 
\emph{orbifold morphism} $F:X\to S_{(s,\bar m)}$ if for all $P\in S$ the divisor $F^*(P)$ is a 
$\varphi(P)$-multiple. The orbifold $S_{(s,\bar m)}$ is said to be \emph{maximal} (with respect to $F$) 
if $F(X)=S\setminus \Sigma_0$ and for all $P\in F(X)$ the divisor $F^*(P)$ is not an $m$-multiple for any $m > \varphi(P)$. 
\end{dfn}

The following result is well known (see for instance~\cite[Prop. 1.4]{ACM-multiple-fibers})
\begin{rem}\label{rem:inducedorb}
Let $F:X\to S_{(s,\bar m)}$ be an orbifold morphism. Then, $F$ induces a morphism
$
F_*:\pi_1(X)\to\pi_1^{\orb}(S_{(s,\bar m)}).
$
Moreover, if the generic fiber of $F$ is connected, then $F_*$ is surjective.
\end{rem}

The language of admissible maps and orbifold morphisms is useful when studying the fundamental group of plane fiber-type curves, as seen by the following result, which is well known in different settings (cf. \cite[Lemma 4.2]{Catanese-Fibred},
\cite[Corollary 2.22]{ji-Eva-orbifold}). Indeed, given a pencil $F$ satisfying Condition~\ref{cond} and a finite set $B\subset \PP^1$ of $s$ elements, one may apply it to $U:=\PP^2\setminus C_B$ to get information about the fundamental group of the complement of $C_B$.

\begin{lemma}[{\cite[Cor. 2.32]{ji-Eva-GOMP}}]\label{lem:exact}
Let $U$ be a smooth quasi-projective variety, let $S$ be a smooth projective curve of genus
	$g$ and let $S'$ be the complement in $S$ of $s$ points, where $s\geq 0$. Let $F:U\to S'$ be an admissible map.
	Let $S_{(s,\bar m)}$ be the maximal orbifold structure of $S$ with respect to $F$ for some
	$\bar m=(m_1,\dots,m_n)$. Then, the following sequence is exact
	$$
	\pi_1(F^{-1}(P))\to\pi_1(U)\xrightarrow{F_*}\pi_1^{\orb}(S_{(s,\bar m)})\to 1,
	$$
	where $P\in S'\setminus B_F$, and the first arrow is induced by the inclusion.
\end{lemma}

\subsection{Fundamental groups of fiber-type curve complements}\label{sec:fundamentalgroupsfiber}
We start by recalling some general facts about fundamental groups of complements of fiber-type curves which were found by the authors in \cite{ji-Eva-orbifold} by exploiting the exact sequence from Lemma~\ref{lem:exact}, as opposed to using braid monodromy arguments.

Given an admissible map $F:U\to S$ from a smooth quasi-projective variety to a smooth projective curve, we want to obtain information regarding a presentation of the fundamental group of $U_B:=U\setminus F^{-1}(B)$. After a Lefschetz-type argument, we may assume that $U$ is a smooth quasi-projective surface, and that $F^{-1}(P)$ is a curve for all $P\in B$, so $U_B$ is the complement of a fiber-type curve in $U$.

Suppose that $B\subset S$ is a finite set. If $B\cup B_F\neq\emptyset$, then $S\setminus (B\cup B_F)$ is aspherical, as it is homotopy equivalent to a wedge of circles. Hence, the long exact sequence in homotopy groups associated to the fibration $F:U_{B\cup B_F}\twoheadrightarrow S\setminus (B\cup B_F)$ yields a short exact sequence
$$
1\to \pi_1(F^{-1}(P))\to \pi_1(U_{B\cup B_F})\to \pi_1(S\setminus (B\cup B_F))\to 1
$$
which splits, since $\pi_1(S\setminus (B\cup B_F))$ is a free group. This yields a description of $\pi_1(U_{B\cup B_F})$ as a semidirect product. Lemma~\ref{lemma:meridians-generators} can be used to obtain a presentation of $\pi_1(U_{B})$
from such a presentation of $\pi_1(U_{B\cup B_F})$ by adding meridians about the irreducible components of
$F^{-1}(B_F\setminus B)$ to the relations, resulting after some work in the following.

\begin{lemma}[{\cite[Lem. 4.1]{ji-Eva-orbifold}}]\label{lem:presentationGeneral}
Let $U$ be a smooth complex quasi-projective variety, let $S$ be a smooth complex projective curve of genus $g$,
and let $F:U\to S$ be an admissible map. Let $B=\{P_1,\ldots, P_{s+1}\}$ be a non-empty finite subset of $S$
with $s+1$ elements, let $U_B:=U\setminus F^{-1}(B)$ and let $S_{(s+1,\bar m)}$ be the maximal orbifold structure on $S$ with
respect to the map $F_|:U_B\to S\setminus B\subset S$, where $\bar m=(m_1,\ldots,m_n)$. For all $i=1,\ldots, n$, let
$Q_i\in S\setminus B$ be the point of multiplicity $m_i$ in the orbifold structure $S_{(s+1,\bar m)}$. Let $K_B$ be
the kernel of
$$
F_*:\pi_1(U_B)\twoheadrightarrow \pi_1^{\orb}\left(S_{(s+1,\bar m)}\right).
$$
Then, $\pi_1(U_{B})$ has a
	presentation with a finite set of generators
	$\SF\cup\SB$, where
	$\SB=\{\gamma_1,\ldots,\gamma_{n+s},\ldots,\gamma_{s+2g+n}\}$, and relations
	\begin{enumerate}
		\enet{\rm(R\arabic{enumi})}
		\item\label{R1}
		$[\gamma_{l},w]=z_{l,w}$, for all $l=1,\ldots,s$ and all $w\in\SF$, where $z_{l,w}$ is a word in~$\SF$
		(that depends both on $\gamma_l$ and $w$), which is the trivial word if $P_l\in\PP^1\setminus B_F$,
		\item\label{R2}
		$[\gamma_{l},w]=z_{l,w}$, for all $l=s+1,\ldots,s+2g+n$ and all $w\in\SF$, where $z_{l,w}$ is a word in~$\SF$
		(that depends both on $\gamma_l$ and $w$).
		\item\label{R3}
		$y=1$ for a finite number of words $y$ in $\SF$,
		\item\label{R4} 
		$z_{i}=\gamma_{2g+s+i}^{m_{i}}$, for any $i=1,\ldots,n$, where $z_{i}$ is a word in~$\SF$.
	\end{enumerate}
	Furthermore, the loops in the generating set satisfy the following:
	\begin{itemize}
		\item The set of loops in $\SF$ generates $K_B$.
		\item $F_*(\SB)$ is a set of generators of
		$\pi_1^{\orb}(S_{(s+1,\bar m)})\cong \FF_{2g+s}*(\ast_{i=1}^n\ZZ_{m_i})$.
		\item $F_*(\gamma_{j})$ is a positively oriented meridian about $P_j$ for $j=1,\ldots,s$ and $F_*(\gamma_{2g+s+i})$ is a positively oriented meridian about the point $Q_i$ of multiplicity
		$m_i$, in a way in which $F_*(\gamma_{j})$ generates the $j$-th $\ZZ$ in the free product $\FF_{2g+s}*(\ast_{i=1}^n\ZZ_{m_i})$ for $j=1,\ldots,2g+s$
		and $F_*(\gamma_{2g+s+i})$ generates the $\ZZ_{m_i}$ factor for $i=1,\ldots,n$.
		\item Let $i\in\{1,2,\ldots,s\}$. If $P_i\in B\setminus M_F$ satisfies that $C_{P_i}:=F^{-1}(P_i)$ is irreducible, then $\gamma_i$ is a positively oriented meridian about $C_{P_i}$.
		\end{itemize}
		
\end{lemma}

\begin{rem}\label{rem:genericCommute}
Under the same notation and assumptions as in Lemma~\ref{lem:presentationGeneral}, let $P\in B$, and suppose that $P\notin B_F$. Then, by the description of $\pi_1(U_B)$ as a quotient of the semidirect product $\pi_1(U_{B\cup B_F})$, any meridian about $F^{-1}(P)$ commutes with all the elements in $K_B$.
\end{rem}

		Now, we turn to the case in which $B$ consists on a single point. 
		
\begin{lemma}\label{lem:kernel1fiber}
With the same notation and assumptions as in Lemma~\ref{lem:presentationGeneral}, suppose moreover that $B=\{P\}$ for some $P\in S\setminus B_F$, that $F_*:\pi_1(U)\to \pi_1^{\orb}\left(S_{(0,\bar m)}\right)$ is an isomorphism and that $M_F\neq \emptyset$ (i.e. $n\geq 1$). Let $K_P$ be the kernel of
$$
F_*:\pi_1\left(U_{\{P\}}\right)\twoheadrightarrow \pi_1^{\orb}\left(S_{(1,\bar m)}\right).
$$
Then,
\begin{itemize}
\item $K_P$ is an abelian group.
\item If $U$ is simply connected, then $K_P$ is a central subgroup of $\pi_1\left(U_{\{P\}}\right)$.
\item $\pi_1\left(U_{\{P\}}\right)$ has a presentation as in Lemma~\ref{lem:presentationGeneral} with $\SB=\{\gamma_1,\ldots,\gamma_{2g},\gamma_{2g+1},\ldots,\gamma_{2g+n}\}$, such that $K_P=\langle \gamma_{2g+1}^{m_1},\ldots,\gamma_{2g+n}^{m_n}\rangle$.
\end{itemize}
\end{lemma}

\begin{proof}
The first and last points follow from \cite[Lem. 4.2]{ji-Eva-orbifold}. The second point follows from step (3) in the proof of \cite[Lem. 4.2]{ji-Eva-orbifold} and from Lemma~\ref{lemma:meridians-generators}.
\end{proof}

\begin{rem}\label{rem:sc}
Note that the hypothesis that $F_*:\pi_1(U)\to \pi_1^{\orb}\left(S_{(0,\bar m)}\right)$ is an isomorphism is automatically satisfied in Lemma~\ref{lem:kernel1fiber} if $F:\PP^2\dashrightarrow S=\PP^1$ is a pencil satisfying Condition~\ref{cond} and $U=\PP^2\setminus\mathcal B$, where $\mathcal B$ is the set of base points of the pencil $F$. Also note that, if $U=\PP^2\setminus\mathcal B$, then $U$ is simply connected. Furthermore, in that case, $U_{\{P\}}=\PP^2\setminus C_P$.
\end{rem}

We now further impose the hypothesis that $U$ is simply connected, in which case $S$ must be $\PP^1$, and thus $g=0$. Moreover, by \cite[Remark 2.2]{ji-Eva-orbifold}, $\# M_F\leq 2$ in that case.

\begin{lemma}[{\cite[Cor. 4.7]{ji-Eva-orbifold}}]\label{lem:presentation1fiber}
Under the same notation and assumptions of Lemma~\ref{lem:kernel1fiber}, suppose moreover that $U$ is simply connected. Then,
$\pi_1(\PP^2\setminus C_P)$ has the following presentation:
	\begin{itemize}
		\item If $\# M_F=2$, and $\bar m=(p,q)$,
		$$
		\pi_1(U_{\{P\}})\cong \langle \gamma_1,\gamma_2: [\gamma_1,\gamma_2^{q}], [\gamma_2,\gamma_1^{p}], \gamma_1^{pn_1}\gamma_2^{qn_2},\gamma_2^{qn_3}\rangle
		$$
		for some $n_1,n_2,n_3\in\ZZ$.
		\item If $\# M_F=1$, and $\bar m=(p)$, then
		$$
		\pi_1(U_{\{P\}})\cong \langle \gamma_1:  \gamma_1^{pn_1}\rangle
		$$
		for some $n_1\in\ZZ$. In particular, $\pi_1(U_{\{P\}})$ is abelian.
	\end{itemize}
\end{lemma}

Let $U$ be a smooth quasi-projective variety, let $S$ be a smooth projective curve of genus $g\geq 0$ and let
$F:U\to S$ be an admissible map. Let $B\subset S$ be a finite set with $s\geq 1$ elements, and consider the
hypersurface of fiber-type $F^{-1}(B)$, and its complement $U_B:=U\setminus F^{-1}(B)$. The
last goal of this section is to describe the changes to the kernel of the map
$F_*:\pi_1(U_B)\to\pi_1^{\orb}(S_{(s,\bar m)})$ after adding points to $B$/removing points from $B$, where
$S_{(s,\bar m)}$ is the maximal orbifold structure on $S$ with respect to $F:U_B\to S\setminus B\subset S$.
If the starting kernel was trivial, this was done in~\cite[Thm. 1.4]{ji-Eva-orbifold}
(addition) and in ~\cite[Thm. 4.9]{ji-Eva-orbifold} (deletion). Theorem~\ref{thm:additiongeneral} below provides a
generalization of~\cite[Thm. 1.4]{ji-Eva-orbifold}, while Theorem~\ref{thm:deletion} below provides a generalization
of ~\cite[Thm. 4.9]{ji-Eva-orbifold}. Before we prove those results, let us introduce the following notation.

\begin{dfn}\label{def:purelycentral}
	We say that a normal subgroup $K$ of a group $G$ is \emph{purely central} if $K\subset Z(G)$ ($K$ is contained
	in the center of $G$) and $K\cap G'=\{1\}$ (the intersection of $K$ with the commutator subgroup
	of $G$ is trivial). Moreover, we call a central extension
	$1\to K \rightmap{i}\ G \rightmap{j}\ H\to 1$ of $K$ by $H$ \emph{purely central} if $i(K)$ is
	purely central in~$G$.
\end{dfn}
Note that a central extension $1\to K \rightmap{i}\ G \rightmap{j}\ H\to 1$ is purely central if and only if it induces a short
exact sequence $1\to K/K' \rightmap{i}\ G/G' \rightmap{j}\ H/H'\to 1$ on the respective abelianizations.
\begin{thm}[Addition theorem]\label{thm:additiongeneral}
	Let $U$ be a smooth quasi-projective variety, let $S$ be a smooth projective curve of genus $g\geq 0$ and let $F:U\to S$ be an admissible map. Let $A\subset S$ be a non-empty set, with $\# A=s\geq 1$, and let $P\in S\setminus A$. Let $U_A:=U\setminus F^{-1}(A)$ (resp. $U_{A\cup\{P\}}:=U\setminus F^{-1}(A\cup\{P\})$). Let $K_A$ be the kernel of
	$$
	F_*:\pi_1(U_A)\to\pi_1^{\orb}(S_{(s,\bar m)}),
	$$
	where $S_{(s,\bar m)}$ is the maximal orbifold structure with respect to $F:U_A\to S\setminus A$. Similarly, let $K_{A\cup\{P\}}$ be the kernel of
	$$
	F_*:\pi_1(U_{A\cup\{P\}})\to\pi_1^{\orb}(S_{(s+1,\bar m)}).
	$$
	Suppose moreover that $P\notin B_F$. Then,
	\begin{enumerate}
	\item\label{item:iso} The epimorphism $\pi_1(U_{A\cup\{P\}})\to \pi_1(U_A)$ induced by inclusion yields an isomorphism between  $K_{A\cup\{P\}}$ and $K_A$.
	\item\label{item:ext} If $$1\to K_A\to \pi_1(U_A)\xrightarrow{F_*}\pi_1^{\orb}(S_{(s,\bar m)})\to 1$$ is a central (resp. purely central) extension, so is $$1\to K_{A\cup\{P\}}\to \pi_1(U_{A\cup\{P\}})\xrightarrow{F_*}\pi_1^{\orb}(S_{(s+1,\bar m)})\to 1,$$ where $S_{(s+1,\bar m)}$ is the maximal orbifold structure on $S$ with respect to $F:U_{A\cup\{P\}}\to S\setminus(A\cup\{P\})$.
	\end{enumerate}
\end{thm}

\begin{proof}
	Let $\bar m=(m_1,\ldots,m_n)$, $n\geq 0$.  Let $A=\{P_2,\ldots, P_{s+1}\}$, and let $B=A\cup\{P\}=\{P_1:=P,P_2,\ldots,P_{s+1}\}$. Consider the presentation of $\pi_1(U_B)$ with generators $\SF\cup \SB$ given by Lemma~\ref{lem:presentationGeneral}, where $\gamma_1\in\SB$ is a meridian about the irreducible curve $F^{-1}(P)$. By Lemma~\ref{lemma:meridians-generators}, one obtains a presentation of $\pi_1(U_A)$ by adding $\gamma_1$ to the set of relations and using that to remove $\gamma_1$ from the set of generators. This gives a presentation of $\pi_1(U_A)$ with generators $\SF\cup\SB\setminus\{\gamma_1\}$. Moreover, since $P\notin B_F$, if $R$ is set of relations of the presentation of $\pi_1(U_B)=\pi_1(U_{A\cup\{P\}})$, then the set of relations of $\pi_1(A)$ is $R\setminus \{[\gamma_1,w]\mid w\in\SF\}$. Note that no word in $R\setminus \{[\gamma_1,w]\mid w\in\SF\}$ contains the letter $\gamma_1$.

Let $p:\pi_1(U_{A\cup\{P\}})\to\pi_1(U_A)$ be the projection induced by the inclusion $U_{A\cup\{P\}}\hookrightarrow U_A$, which, at the level of fundamental groups, corresponds to the quotient by the normal subgroup generated by $\gamma_1$. Using the presentations that we considered for both $\pi_1(U_{A\cup\{P\}})$ and $\pi_1(U_A)$, we obtain that $p$ splits via a monomorphism $\sigma: \pi_1(U_A)\hookrightarrow \pi_1(U_{A\cup\{P\}})$ which takes the generating set $\SF\cup\SB\setminus\{\gamma_1\}$ to itself via the identity.

By Lemma~\ref{lem:exact}, Lemma~\ref{lem:presentationGeneral}, and the fact that $S_{(s+1,\bar m)}$ is the maximal orbifold structure on $S$ with respect to $F:U_{A\cup\{P\}}\to S\setminus(A\cup\{P\})$, we get that $K_{A\cup\{P\}}$ (resp. $K_A$) is the subgroup of $\pi_1(U_{A\cup\{P\}})$ (resp. $\pi_1(U_A)$) generated by the subset of generators $\SF$ of $\pi_1(U_{A\cup\{P\}})$ (resp. $\pi_1(U_A)$). Hence, $\sigma$ induces an epimorphism $\sigma_|:K_A\to K_{A\cup\{P\}}$. Since $\sigma$ is injective, $\sigma_|$ is an isomorphism and $p_|:K_{A\cup\{P\}}\to K_A$ is its inverse. This concludes the proof of part \eqref{item:iso}.

%
	
	Let us prove part~\eqref{item:ext}. Suppose that $K_A$ is a central subgroup of $\pi_1(U_A)$, or in other words, every element of $\SF$ commutes with every element of $\SF\cup\SB\setminus\{\gamma_1\}$. Using $\sigma$ and the fact that $\gamma_1$ commutes with every element of $\SF$ in $\pi_1(U_{B\cup\{P\}})$, we obtain that $K_{A\cup\{P\}}$ is a central subgroup of $\pi_1(U_{A\cup\{P\}})$.
	
%
	
	Finally, suppose that $K_A$ is a purely central subgroup. Let us show that the intersection of $K_{A\cup\{P\}}$ with the derived subgroup of $\pi_1(U_{A\cup\{P\}})$ is trivial. Note that the isomorphism $p_|:K_{A\cup\{P\}}\to K_A$ takes the intersection of $K_{A\cup\{P\}}$ with the derived subgroup of $\pi_1(U_{A\cup\{P\}})$ to the intersection of $K_A$ with the derived subgroup of $\pi_1(U_{A})$, which is trivial by hypothesis. This concludes the proof of part~\eqref{item:ext}.
\end{proof}

We will not need the following theorem in the proof of the Main Theorem~\ref{thm:newgeneric}, but we include it as a natural counterpart to Theorem~\ref{thm:additiongeneral}.
\begin{thm}[Deletion]\label{thm:deletion}
	Let $U$ be a smooth quasi-projective variety, let $S$ be a smooth projective curve and let $F:U\to S$ be an admissible map. Let $B\subset S$ be a non-empty set with $\# B=s\geq 1$ and let $P\in B$. Denote by $U_B:=U\setminus F^{-1}(B)$ (resp. $U_{B\setminus\{P\}}:=U\setminus F^{-1}(B\setminus\{P\})$). Let $K_B$ be the kernel of
	$$
	F_*:\pi_1(U_B)\to\pi_1^{\orb}(S_{(s,\bar m)}),
	$$
	where $S_{(s,\bar m)}$ is the maximal orbifold structure with respect to $F:U_B\to S\setminus B$. Let $K_{B\setminus\{P\}}$ the kernel of
	$$
	F_*:\pi_1(U_{B\setminus\{P\}})\to\pi_1^{\orb}(S_{(s-1,\bar m')}),
	$$
	where $S_{(s-1,\bar m')}$ is the maximal orbifold structure with respect to $F:U_{B\setminus\{P\}}\to S\setminus (B\setminus\{P\})$.
	
	Then, $K_{B\setminus\{P\}}$ is a quotient of $K_B$ by the restriction of the epimorphism $p:\pi_1(U_B)\to\pi_1(U_{B\setminus\{P\}})$ induced by inclusion.
	
	In particular, if $K_B$ is finite/cyclic/central in $\pi_1(U_B)$/purely central in $\pi_1(U_{B})$, then $K_{B\setminus\{P\}}$ is finite/cyclic/central in $\pi_1(U_{B\setminus\{P\}})$/purely central in $\pi_1(U_{B\setminus\{P\}})$ respectively.
\end{thm}

\begin{proof}
	Let $Q\in S\setminus (B\cup B_F)$. The commutativity of the diagram
	$$
	\begin{tikzcd}
	F^{-1}(Q)\arrow[r,hook]\arrow[d,"\mathrm{Id}"] & U_{B}\arrow[d,hook]\\
	F^{-1}(Q)\arrow[r,hook] & U_{B\setminus\{P\}}
	\end{tikzcd}
	$$
	and the exactness of the sequences of Lemma~\ref{lem:exact} corresponding to $U_{B}$ and $U_{B\setminus\{P\}}$ imply that $p$ restricts to an epimorphism $p_|:K_B\to K_{B\setminus\{P\}}$, and the result follows.
\end{proof}

\section{Characterization of Oka's $G(p;q;r)$ groups}\label{s:Okagroups}

Oka shows that the groups from Definition~\ref{def:Gpqr} have the following first three properties, while the last one is a
consequence of the presentation in equation \eqref{eq:presGpqr}.
\begin{prop}[Theorem 2.12 in \cite{Oka-genericRjoin}]\label{prop:center}
	Let $s=\gcd(p,q)$ and let $a=\gcd\left(\frac{q}{s}, r \right)$. Then,
	\begin{enumerate}
		\item\label{part:center} The center $Z(G(p;q;r))$ of $G(p;q;r)$ contains the finite cyclic group  $\ZZ_{\frac{r}{a}}$.
		\item\label{part:purely}  $\ZZ_{\frac{r}{a}}$ is purely central in $G(p;q;r)$.
		\item\label{part:proj} The quotient group $G(p;q;r)/\ZZ_{\frac{r}{a}}$ is isomorphic to $\FF_{s-1}*\ZZ_{\frac{p}{s}}*\ZZ_a$, where $\frac{p}{s}$ and $a$ are coprime.
		\item The quotient group $G(p;q;r)/G(p;q;r)'$ is isomorphic to $\ZZ_{r\cdot\frac{p}{s}}\times \ZZ^{s-1}$, where $G(p;q;r)'$ is the commutator subgroup of $G(p;q;r)$.
	\end{enumerate}
\end{prop}

We now show that the properties in Proposition~\ref{prop:center} completely determine the isomorphism class of the group $G(p;q;r)$. This is the main result of this section, and it will be a key element in proving the Main Theorem~\ref{thm:newgeneric}.

\begin{thm}\label{thm:OkaUnique}
	Let $n$, $m_1$ and $m_2$ be positive integers, where $\gcd(m_1,m_2)=1$. Let $r$ be a non-negative integer. Let
	$$
	H:=\langle x_1,\ldots,x_r,a,b: a^{m_1}=b^{m_2}, a^{m_1n}=1, [a^{m_1},x_i]=1\,\forall\,i=1,\ldots,r\rangle.
	$$
	Up to isomorphism, $H$ is the only group  such that the following hold:
	\begin{enumerate}
		\item\label{item:central} There exists a central extension of the form
		$$
		1\to \ZZ_n \to H\to \FF_{r}*\ZZ_{m_1}*\ZZ_{m_2}\to 1.
		$$
		\item\label{item:quotient} The quotient $H/H'$ is isomorphic to $\ZZ_{n m_1 m_2}\times\ZZ^r$,  where $H'$ is the commutator subgroup of $H$.
	\end{enumerate}
	In particular, $H\cong G((r+1)m_1;(r+1)m_2;nm_2)$, and, if $n$ is coprime with $m_1m_2$, then $H\cong \ZZ_n\times \left(\FF_{r}*\ZZ_{m_1}*\ZZ_{m_2}\right)$.
\end{thm}

\begin{proof}
	Let us show that the sequence
	\begin{equation}\label{eq:sesH}
		\begin{array}{cccccclcl}
			1&\to & \ZZ_n & \to & H & \to &\FF_r*\ZZ_{m_1}*\ZZ_{m_2}&\to& 1\\
			& & 1&\mapsto &a^{m_1} \\
			& & & & a &\mapsto & t\\
			& & & & b &\mapsto & u\\
			& & & & x_i &\mapsto & y_i,\quad  i=1,\ldots, r\\
		\end{array}
	\end{equation}
	is exact, where $y_1,\ldots,y_r$ are free generators of the $\FF_r$ free factor, $t$ is a generator of the $\ZZ_{m_1}$ free factor, and $u$ is a generator of the $\ZZ_{m_2}$ free factor. By the presentation of $H$, the order of $a$ divides $m_1n$, and it is exactly $m_1n$ because that is the order of $a$ in the abelianization of $H$, which is isomorphic to $\ZZ_{nm_1m_2}\times\ZZ^r$. Hence, the map $\ZZ_n\to H$ is injective, and its image is the subgroup generated by $a^{m_1}$, which is central. The induced map $H/\langle a^{m_1}\rangle\to \FF_r*\ZZ_{m_1}*\ZZ_{m_2}$ is an isomorphism, so the sequence \eqref{eq:sesH} is exact. Hence, $H$ is a group satisfying conditions~\eqref{item:central} and~\eqref{item:quotient} in the statement of this proposition.
	
	Let $G$ be a group satisfying those same two conditions. In particular, there exists a central extension of the form
	\begin{equation}\label{eq:sesG}
		1\to \ZZ_n\xrightarrow{\iota} G\xrightarrow{\psi}\FF_r*\ZZ_{m_1}*\ZZ_{m_2}\to 1.
	\end{equation}
	
	Consider the subgroup $\psi^{-1}(\ZZ_{m_1})$. The short exact sequence \eqref{eq:sesG} induces the short exact sequence in the top row of the following commutative diagram. The vertical arrows are inclusions followed by the abelianization morphism, and the bottom row is exact by condition~\eqref{item:quotient} (and the fact that condition~\eqref{item:central} and~\eqref{item:quotient} together imply that the extension in equation~\eqref{eq:sesG} is purely central).
	$$
	\begin{tikzcd}
		1\arrow[r]& \ZZ_n\arrow[r,"\iota"]\arrow[d,"\mathrm{Id}_{\ZZ_n}"]& \psi^{-1}(\ZZ_{m_1})\arrow[r, "\psi"]\arrow[d,hook]& \ZZ_{m_1}\arrow[r]\arrow[d,hook] & 1\\
		1\arrow[r]& \ZZ_n\arrow[r,"\iota"]\arrow[d,"\mathrm{Id}_{\ZZ_n}"]& G\arrow[r, "\psi"]\arrow[d]& \FF_r*\ZZ_{m_1}*\ZZ_{m_2}\arrow[r]\arrow[d] & 1\\
		1\arrow[r]& \ZZ_n\arrow[r]& \ZZ_{nm_1m_2}\times\ZZ^r\arrow[r]& \ZZ_{m_1m_2}\times\ZZ^r\arrow[r] & 1\\
	\end{tikzcd}
	$$
	Note that the composition of the vertical arrows on the right is injective, as is the composition of the vertical arrows on the left. Hence, $\psi^{-1}(\ZZ_{m_1})$ is a finite group of order $nm_1$ which embeds into $\ZZ_{nm_1m_2}\times\ZZ^r$, and therefore it is cyclic. Let $\hat a$ be a generator of $\psi^{-1}(\ZZ_{m_1})$. Similarly, $\psi^{-1}(\ZZ_{m_2})$ is cyclic of order $nm_2$. Note that $\langle \hat a\rangle \cap \psi^{-1}(\ZZ_{m_2}) =\langle \hat a^{m_1}\rangle=\iota(\ZZ_n)$, so in particular $\hat a^{m_1}$ is a central element. By a similar argument, it is possible to pick a generator $\hat b$ of $\psi^{-1}(\ZZ_{m_2})$ such that $\langle \hat b^{m_2}\rangle=\iota(\ZZ_n)$. Consider the epimorphism $\psi^{-1}(\ZZ_{m_2})\twoheadrightarrow \iota(\ZZ_n)$ sending $\hat b$ to $\hat b^{m_2}$. Note that, after suitable group isomorphisms in both the domain and the target, this epimorphism is just the natural ring projection $\ZZ_{m_2n}\twoheadrightarrow\ZZ_n$, which induces an epimorphism between the groups of units (in this case, the set of cyclic generators of the additive group) by the Chinese Remainder Theorem. Hence, it is possible to pick $\hat b$ in a way such that $\hat b^{m_2}=\hat a^{m_1}$, and we assume this from now on.
	
	Let $\hat x_i$ be an element of $G$ such that $\psi(\hat x_i)=y_i$ for all $i=1,\ldots, r$. There is a unique group homomorphism $\phi:H\to G$ sending $a$ to $\hat a$, $b$ to $\hat b$ and $x_i$ to $\hat x_i$ for all $i=1,\ldots,r$. Moreover, $\phi$ is an isomorphism by the $5$-lemma applied to the following commutative diagram, where the top and bottom rows are the short exact sequences in equations \eqref{eq:sesH} and \eqref{eq:sesG} respectively:
	$$
	\begin{tikzcd}
		1\arrow[r]& \ZZ_n\arrow[r]\arrow[d,"\mathrm{Id}_{\ZZ_n}"]& H\arrow[r]\arrow[d, "\phi"]& \FF_r*\ZZ_{m_1}*\ZZ_{m_2}\arrow[r]\arrow[d, "\mathrm{Id}_{\FF_r*\ZZ_{m_1}*\ZZ_{m_2}}"] & 1\\
		1\arrow[r]& \ZZ_n\arrow[r,"\iota"]& G\arrow[r, "\psi"]& \FF_r*\ZZ_{m_1}*\ZZ_{m_2}\arrow[r] & 1
	\end{tikzcd}
	$$
	
\end{proof}

\begin{exam}\label{ex:cyclic}
	Theorem~\ref{thm:OkaUnique} implies that, for every pair of positive integers $m,k$, then $G(m,1,k)\cong G(1,1,mk)\cong G(mk,1,1)\cong\ZZ_{mk}$.
\end{exam}

The following is a consequence of Proposition~\ref{prop:center} and  Theorem~\ref{thm:OkaUnique}.
\begin{cor}[The isomorphism problem for $G(p;q;r)$]\label{cor:isoGpqr}
	Let $p,q,r$ be positive integers, let $s=\gcd(p,q)$ and let $a=\gcd\left(\frac{q}{s},r\right)$. Then,
	\begin{equation}\label{eq:isoGpqr}
		G(p;q;r)\cong G(p;as;r)\cong G\left(as;p;\frac{r}{a}\cdot\frac{p}{s}\right).
	\end{equation}
	In particular, there exists a representative $G(p';q';r')$ of the isomorphism class of
	$G(p;q;r)$ with $p'\geq q'$ and  $\frac{q'}{\gcd(p',q')}\mid r'$, namely one of the last two
	groups in equation~\eqref{eq:isoGpqr}. Moreover, if $G(p;q;r)$ is not finite cyclic, such representative is unique.
\end{cor}
\begin{proof}
	The part of the statement before the word ``moreover'' follows directly from Proposition~\ref{prop:center} and Theorem~\ref{thm:OkaUnique}.
	
	If $G(p;q;r)$ satisfies that $\FF_{s-1}*\ZZ_{\frac{p}{s}}*\ZZ_a$ as in Proposition~\ref{prop:center}\eqref{part:proj} is not cyclic, then $G(p;q;r)$ is not abelian and $\ZZ_{\frac{r}{a}}$ as in Proposition~\ref{prop:center}\eqref{part:center} is the center of $G(p;q;r)$. If $G(p;q;r)$ satisfies that $\FF_{s-1}*\ZZ_{\frac{p}{s}}*\ZZ_a$ as in Proposition~\ref{prop:center}\eqref{part:proj} is infinite cyclic, then $\ZZ_{\frac{r}{a}}$ as in Proposition~\ref{prop:center}\eqref{part:center} is the torsion subgroup of $G(p;q;r)\cong \ZZ_{\frac{r}{a}}\times\ZZ$ by Theorem~\ref{thm:OkaUnique} (note that $G(p;q;r)$ is abelian in that case). Lastly, if $G(p;q;r)$ satisfies that $\FF_{s-1}*\ZZ_{\frac{p}{s}}*\ZZ_a$ as in Proposition~\ref{prop:center}\eqref{part:proj} is finite cyclic, then $G(p;q;r)$ is also finite cyclic by Theorem~\ref{thm:OkaUnique}. Hence, if $G(p;q;r)$ is not finite cyclic, then $\ZZ_{\frac{r}{a}}$ is a characteristic subgroup of $G(p;q;r)$.
	
	For $i=1,2$, let $p_i,q_i,r_i$ be positive integers, let $s_i=\gcd(p_i,q_i)$, and let $a_i=\gcd\left(\frac{q_i}{s_i}, r_i\right)$. Suppose that $
	G(p_1;q_1;r_1)\cong G(p_2;q_2;r_2)
	$ and that those groups are not finite cyclic. By Proposition~\ref{prop:center},  Theorem~\ref{thm:OkaUnique}, and the discussion in the previous paragraph,
	$$
	\left(s_1,\frac{r_1}{a_1},S_1\right)=\left(s_2,\frac{r_2}{a_2},S_2\right),
	$$
	where $S_i=\left\{a_i,\frac{p_i}{s_i}\right\}$, $i=1,2$. The uniqueness result follows from this. 
\end{proof}

In Example~\ref{ex:Oka} we saw that the group $G(p;q;r)$ can be realized as the fundamental group of a plane fiber-type curve complement for all $p,q,r\geq 1$. In fact, the only plane curves that may have $G(p;q;r)$ as the fundamental group of their complements are fiber-type curves, as seen in the following result. For simplicity with the notation, we write $\pi_1^{\orb}\left(\PP^1_{\left(s,\left(\frac{p}{s},a\right)\right)}\right)$, even though the coordinates of $\left(\frac{p}{s},a\right)$ may equal to $1$, to mean that the multiplicities $\geq 2$ in the orbifold structure are the coordinates of $\left(\frac{p}{s},a\right)$ which are $\geq 2$.

\begin{prop}\label{prop:GpqrImpliesFiber}
Let $C\subset \PP^2$ be a curve such that $\pi_1(\PP^2\setminus C)$ is isomorphic to $G(p;q;r)$. Then, $C$ is a fiber-type curve
associated to a pencil $F$ satisfying Condition~\ref{cond} such that the kernel of
$$
F_*:\pi_1(\PP^2\setminus C)\to \pi_1^{\orb}\left(\PP^1_{\left(s,\left(\frac{p}{s},a\right)\right)}\right)\cong \FF_{s-1}*\ZZ_{\frac ps}*\ZZ_a
$$
is a purely central subgroup of $\pi_1(\PP^2\setminus C)$ which is isomorphic to $\ZZ_{\frac ra}$, where $s=\gcd(p,q)$,
$a=\gcd\left(\frac qs, r\right)$ and $\PP^1_{\left(s,\left(\frac{p}{s},a\right)\right)}$ is the maximal orbifold structure
on $\PP^1$ with respect to $F:\PP^2\setminus C\to\PP^1$. Moreover, the irreducible components of $C$ are distinct fibers
of the pencil~$F$.
\end{prop}

\begin{proof}
By Lemma~\ref{lem:coprime} and Proposition~\ref{prop:center}, $C$ has $s$ irreducible components and the greatest
common divisor of the degrees of its components is $r\cdot\frac{p}{s}$.

If $C$ is irreducible ($s=1$) and $m:=\max\left\{a,p\right\}>\min\left\{a,p\right\}=1$,
then Theorem~\ref{thm:OkaUnique} implies that $\pi_1(\PP^2\setminus C)\cong \ZZ_{m\cdot\frac{r}{a}}$. In that case,
we may construct the pencil $F$ as follows: Let $f$ be an irreducible polynomial defining $C$, which has degree
$m\cdot\frac{r}{a}$, and let $g$ be any irreducible polynomial of degree $\frac{r}{a}$. The pencil
$F=[f:g^{m}]$ has connected generic fibers by \cite[Lem. 2.6]{ji-Eva-Zariski}. Note that $F$ has no more multiple fibers
other than $V(g^m)$ or else the abelian group $\pi_1(\PP^2\setminus C)$ would surject onto a free product of two
non-trivial cyclic groups. Hence, $F$ satisfies Condition~\ref{cond} and the statement follows.

If $s=m=a=p=1$, then $\pi_1(\PP^2\setminus C)\cong \ZZ_{r}$ and $g$ can be any irreducible polynomial of degree
$r$ such that $V(g)$ is transversal to $C$, that is, $\# V(g)\cap C=r^2$. The transversality condition implies that
the pencil $F=[f:g]$ has connected generic fibers and no multiple fibers.

If $C$ has two irreducible components and $\max\left\{\frac p2, a\right\}=1$,
then $\pi_1(\PP^2\setminus C)\cong\ZZ_{\frac{r}{a}}\times \ZZ$ by Theorem~\ref{thm:OkaUnique}. The generalized Albanese
variety of $\PP^2\setminus C$ is $\CC^*$ and the Albanese morphism $\alpha:\PP^2\setminus C\to \CC^*$ serves as our desired
pencil (see the proof of \cite[Lem. 2.26]{ji-Eva-orbifold}).

Now suppose that $C$ does not fall under any of the previous cases, that is, either $s=1$ and $\min\{\frac{p}{s},a\}>1$,
or $s=2$ and $\max\{\frac{p}{s},a\}>1$, or $s>2$. In all of these cases $\FF_{s-1}*\ZZ_{\frac ps}*\ZZ_a$ can be interpreted
as the orbifold fundamental group of a smooth quasi-projective curve of negative orbifold Euler characteristic.
Under such conditions one can apply \cite[Theorem 3.5]{ji-Eva-GOMP} to the epimorphism
$$\psi:\pi_1(\PP^2\setminus C)\cong G(p;q;r)\to \FF_{s-1}*\ZZ_{\frac ps}*\ZZ_a$$
from Proposition~\ref{prop:center}, to guarantee the existence of a surjective algebraic morphism with connected generic
fibers $F:\PP^2\setminus C\to\PP^1\setminus\{s \text{ points}\}$ such that the induced morphism at the level of
(orbifold) fundamental groups
$$
F_*:\pi_1(\PP^2\setminus C)\to \pi_1^{\orb}\left(\PP^1_{\left(s,\left(\frac{p}{s},a\right)\right)}\right)\cong \FF_{s-1}*\ZZ_{\frac ps}*\ZZ_a
$$
coincides with $\psi$ up to isomorphism in the target, where $\PP^1$ is endowed with the maximal orbifold structure
with respect to $F:\PP^2\setminus C\to \PP^1$. The map $F$ determines a pencil $F:\PP^2\dashrightarrow \PP^1$, and
the fibers of $F$ over the $s$ marked points in $\PP^1$ must be contained in $C$. Since $C$ only has $s$ irreducible
components, $C$ is the fiber-type curve corresponding to $F$ and the $s$ marked points in~$\PP^1$.
\end{proof}

We end this section by recalling the following result, which shows that the study of Oka's groups includes all of the
free products of cyclic groups which can be realized as fundamental groups of plane curve complements.

\begin{prop}[{\cite[Cor. 3.14]{ji-Eva-orbifold}}]
Let $C\subset \PP^2$ be a curve. If $\pi_1(\PP^2\setminus C)$ is isomorphic to a free product of cyclic groups, then
$\pi_1(\PP^2\setminus C)\cong G(sp;sq;q)\cong\FF_{s-1}*\ZZ_p*\ZZ_q,$
for some $p,q\in \ZZ_{>0}$ such that $\gcd(p,q)=1$, where $s$ is the number of irreducible components of $C$.
\end{prop}

\section{Main Theorem}\label{s:main}

Let us start by proving the Main Theorem~\ref{thm:newgeneric} in the case when $C_B$ is irreducible (i.e. $s=1$).

\begin{thm}\label{thm:1fiber}
Let $F=[f_{kp}^q:f_{kq}^p]:\PP^2\dashrightarrow \PP^1$ be a pencil satisfying Condition~\ref{cond}. Let $P\in\PP^1\setminus B_F$, and let $C_P=\overline{F^{-1}(P)}$. Then
	$$\pi_1(\PP^2\setminus C_P)\cong G(p;q;kq),$$
	and $F$ induces the following purely central extension
	$$
	1\to \ZZ_{k}\to \pi_1(\PP^2\setminus C_P)\xrightarrow{F_*} \pi_1^{\orb}(\PP^1_{(1,\bar m)})\cong \ZZ_p*\ZZ_q\to 1,
	$$
	where $\PP^1_{(1,\bar m)}$ is the maximal orbifold structure on $\PP^1$ with respect to $F:\PP^2\setminus C_P\to\PP^1\setminus\{P\}\subset\PP^1$.
\end{thm}

In Theorem~\ref{thm:1fiber}, note that $\bar m$ is empty if $p=q=1$, equal to $(p)$ if $p>q=1$, and equal to $(p,q)$ if $q>1$. The proof of Theorem~\ref{thm:1fiber} is done in two separate lemmas. We do it in two separate cases: Lemma~\ref{lem:purelycentral} and Lemma~\ref{lem:genericFiberNoMultiple} prove Theorem~\ref{thm:1fiber} in the cases where $M_F\neq \emptyset$ and $M_F= \emptyset$ respectively. Lemma~\ref{lem:purelycentral} is an extension of \cite[Theorem 4.13]{ji-Eva-orbifold} (which only addresses the case $k=1$).

\begin{lemma}\label{lem:purelycentral}
Theorem~\ref{thm:1fiber} holds if $M_F\neq\emptyset$ (i.e. $p>1$).
%
\end{lemma}

\begin{proof}
Let $K_P$ be the kernel of $F_*:\pi_1(\PP^2\setminus C_P)\to \pi_1^{\orb}\left(\PP^1_{(1,\bar m)}\right)$, where $\PP^1_{(1,\bar m)}$ is maximal with respect to $F_|:\PP^2\setminus C_P\to \PP^1\setminus \{P\}$.

If $\# M_F=1$, then $\bar m=(p)$ and $q=1$, so $\pi_1^{\orb}\left(\PP^1_{(1,\bar m)}\right)\cong\ZZ_p$. Lemma~\ref{lem:presentation1fiber} yields that $\pi_1(\PP^2\setminus C_P)\cong\ZZ_{pn_1}$ for some $n_1\in \ZZ$, and Lemma~\ref{lem:kernel1fiber} yields that $K_P\cong p\cdot \ZZ_{pn_1}\cong \ZZ_{n_1}$. The result follows from this observation.

If $\# M_F=2$, then $\bar m=(p,q)$ and $p,q>1$, so $\pi_1^{\orb}\left(\PP^1_{(1,\bar m)}\right)\cong\ZZ_p*\ZZ_q$. By Lemma~\ref{lem:presentation1fiber},
$$
\pi_1(\PP^2\setminus C_P)\cong \langle \gamma_1,\gamma_2: [\gamma_1,\gamma_2^{q}], [\gamma_2,\gamma_1^{p}], \gamma_1^{pn_1}\gamma_2^{qn_2},\gamma_2^{qn_3}\rangle
		$$
for some $n_1,n_2,n_3\in\ZZ$, and $K_P=\langle \gamma_1^p,\gamma_2^q\rangle$ by Lemma~\ref{lem:kernel1fiber}. Hence, $\pi_1(\PP^2\setminus C_P)$ is a central extension of
	$\ZZ_{p}*\ZZ_{q}$ by $K_P$. Let us show
	that~$K_P\cong\ZZ_k$.
	
	Note that the generic fiber of $F$ in $\PP^2\setminus \mathcal B$ is an irreducible curve of
	degree $kpq$. Abelianizing the presentation of $\pi_1(\PP^2\setminus C_P)$ and using Lemma~\ref{lem:coprime}, we obtain that $k=n_1n_3$ and that $\gcd(n_1p,n_2q,n_3q)=1$. Note that $K_P$ is a quotient of
	$\langle \gamma_1^{p},\gamma_2^{q}\mid [\gamma_1^p,\gamma_2^q], (\gamma_1^{p})^{n_1}(\gamma_2^{q})^{n_2},(\gamma_2^{q})^{n_3}\rangle\cong \ZZ_k$.
	Also, $K_P$ has a quotient $G$ (the image of composition of the inclusion into
	$\pi_1(\PP^2\setminus C_P)$ with the abelianization of $\pi_1(\PP^2\setminus C_P)$) which is
	isomorphic to  $\ZZ_k$: this is true because of the short exact sequence
	$$
	1\to G\to \ZZ_{kpq}\xrightarrow{(F_*)^{\mathrm{ab}}} \ZZ_{pq}\to 1.
	$$
	Since $\ZZ_k$ is finite, this implies both $K_P\cong\ZZ_k$ and that $K_P$ is a purely central subgroup. The result now follows from Theorem~\ref{thm:OkaUnique}.
\end{proof}

Now we address the case where there are no multiple fibers.

\begin{lemma}\label{lem:genericFiberNoMultiple}
	Theorem~\ref{thm:1fiber} holds if $M_F=\emptyset$ (i.e. $p=q=1$).
%
%
\end{lemma}

\begin{proof}
	Let $Q\in\PP^1\setminus (B_F\cup\{P\})$. By Lemma~\ref{lem:exact} applied to $F:\PP^2\setminus C_P\to\PP^1\setminus\{P\}$, the inclusion $\iota: F^{-1}(Q)\hookrightarrow \PP^2\setminus C_P$ induces a surjection $\iota_*:\pi_1(F^{-1}(Q))\twoheadrightarrow \pi_1(\PP^2\setminus C_P)$. By Remark~\ref{rem:genericCommute}, every meridian about $C_P$ commutes with the elements of $\mathrm{Im}(\iota_*)=\pi_1(\PP^2\setminus C_P)$, so the meridians about $C_P$ lie in the center of $\pi_1(\PP^2\setminus C_P)$. Since $\pi_1(\PP^2\setminus C_P)$ is generated by the meridians about $C_P$ by Lemma~\ref{lemma:meridians-generators}, $\pi_1(\PP^2\setminus C_P)$ is abelian, and by Lemma~\ref{lem:coprime}, $\pi_1(\PP^2\setminus C_P)\cong\ZZ_k$. The result follows from Lemma~\ref{lem:exact} and Theorem~\ref{thm:OkaUnique}.
\end{proof}

\begin{rem}
The proof of Lemma~\ref{lem:genericFiberNoMultiple} can be applied to study the fundamental group of the complement of irreducible generic fiber-type curves in smooth simply connected projective surfaces. Indeed, suppose that $X$ is a smooth simply connected projective surface, and let $F:X\dashrightarrow S$ be a dominant rational map to a smooth projective curve $S$. Suppose that $F$ has connected generic fibers, that is, there exists a minimal finite set of points $\mathcal B\subset X$ such that $F:X\setminus \mathcal B\to S$ is an admissible map. Since $X$ is simply connected, $S$ must be $\PP^1$. Let $P\in\PP^1\setminus B_F$, and let $C_P:=\overline F^{-1}(P)$. The process of resolution of indeterminacies yields that $\mathcal B\subset C_P$. By Lemma~\ref{lemma:meridians-generators}, $\pi_1(X\setminus C_P)$ is normally generated by a meridian around $C_P$.

If $M_F=\emptyset$, the proof of Lemma~\ref{lem:genericFiberNoMultiple} yields that $\pi_1(X\setminus C_P)$ is abelian, and, since it is normally generated by one element, it must be cyclic.
\end{rem}

The following result (Corollary~\ref{cor:1or0multipleAbelianintro} in the Introduction) is an immediate consequence of Lemmas~\ref{lem:coprime}, \ref{lem:purelycentral} and~\ref{lem:genericFiberNoMultiple}.

\begin{cor}\label{cor:1or0multipleAbelian}
	If an irreducible curve $C$ is the closure of a generic fiber of a pencil $F:\PP^2\dashrightarrow \PP^1$ of degree $d$ with at most one multiple fiber, then $\pi_1(\PP^2\setminus C)\cong\ZZ_d$.
\end{cor}


\begin{exam}
\label{exam:a17}
As a non-trivial application of Corollary~\ref{cor:1or0multipleAbelian} we will revisit the
Zariski pair of sextics with an $\A_{17}$-singularity found by Artal in 1994 in his seminal paper
on Zariski pairs (cf.~\cite{Artal-couples}). He proved that such equisingular family consisted of
four connected components $\cF_1,...,\cF_4$, two of which $\cF_3$, $\cF_4$ correspond to irreducible
sextics, whereas the other two correspond to two cubics with maximal order of contact. From the former
families, only one of them, say $\cF_3$ consisted of sextics of torus type $(2,3)$. It is well known that
the fundamental groups of the complement of any two curves in a connected component of an equisingular
family are isomorphic. Denote by $G_i$ the group corresponding to the family of sextics in $\cF_i$.

From our discussion above, the group $G_3$ admits a surjection onto $\ZZ_2*\ZZ_3$, hence it is not abelian.
On the other hand, the family $\cF_4$ consists of irreducible sextics of non-torus type. Artal proved that
the fundamental groups $G_3$ and $G_4$ are not isomorphic by showing that their Alexander polynomials
(an invariant of the fundamental group) do not coincide. In 2005, Eyral-Oka
(\cite{Eyral-Oka-sextics}) showed that $G_4$ is abelian (and hence $G_4\cong\ZZ_6$). They did this
using the Zariski-Van Kampen method, which heavily relies on the explicit calculation of the braid
monodromy of the curve to retrieve a presentation of the fundamental group.

We present an alternative more conceptual (and significantly shorter) proof of $G_4\cong\ZZ_6$
using Corollary~\ref{cor:1or0multipleAbelian}.
Using the equation of a curve in $\cF_4$,
$$
f=(xz^2-y^2z-x^2y)^2-4x(xz-y^2)(y(xz-y^2)-x^3)-(y(xz-y^2)-x^3)^2,
$$
(see~\cite[p. 240]{Artal-couples}) one can check that the tangent line $\ell=V(x)$ is such that
$f+\lambda x^6$ is generically an equisingular deformation of $f$. One can do this by calculating
a Milnor non-degenerated Newton polygon at the singular point $P$. This can be achieved by the
transformation $x\to w:=x+y^2+y^5+4y^8$. The local equation $x^6$ becomes $w^6$ whose monomials
belong to the convex hull of the Newton polygon. This implies that the Milnor number of the singularity
of $f+\lambda x^6$ at $P$ does not change for any $\lambda\in\CC$ (\cite{Kouchnirenko-polyedres}) and
hence this is an equisingular family (\cite{Le-Ramanujam-invariance}).
Since $P$ is the only singularity of $C$, this implies that the curve $C$ is a
generic member of the pencil $[f:x^6]$. Moreover, note that $x^6$ is the only multiple fiber or else
the pencil would not have any irreducible fibers, which contradicts $C$ being irreducible.
Corollary~\ref{cor:1or0multipleAbelian} implies the result.
\end{exam}

\begin{exam}
\label{exam:muconstant}
	Consider $f=0$ a germ at $(\CC^2,0)$ and $\mathfrak{m}$ the maximal ideal at the local ring.
	Recall that the \emph{$\mu$-constant determinacy of $f$ at $0$},
	denoted by $\mu$-$\det(f,0)$, is the smallest positive integer $s$ such that $f+\lambda h$ is $\mu$-constant
	in $\lambda\in [0,\varepsilon)$ for any germ $h\in\mathfrak{m}^s$ and some $\varepsilon>0$ depending on $h$.
	Let's assume $C:=V(f)$ is an irreducible curve with at most one singular point $P$ such that
	$\mu$-$\det(f,P)\leq d$, then $\pi_1(\PP^2\setminus C)=\ZZ_d$.
	The condition $\mu$-$\det(f,P)\leq d$ implies that, given a generic line $\ell=V(g)$ passing through $P$,
	the curve $C_\lambda=V(f+\lambda g^d)$ has generically the same local type of singularity at $P$ as $C$.
	Then $C$ is a generic member of the pencil $[f:g^d]$ which contains exactly one multiple fiber.
	Corollary~\ref{cor:1or0multipleAbelian} implies the result. Note that this relaxes the strict inequality
	condition required by Dimca in~\cite[\S4 Prop. 3.11]{Dimca-singularities}.
\end{exam}

%

Now, we are ready to prove the Main Theorem (Theorem~\ref{thm:newgeneric}).

\begin{proof}[Proof of Theorem~\ref{thm:newgeneric}]
The statement about $\pi_1(\PP^2\setminus C_B)$ follows from induction on the size of $B$: the base case is done in Theorem~\ref{thm:1fiber}, and the induction step follows from  Lemma~\ref{lem:coprime} and Theorems~\ref{thm:additiongeneral} and~\ref{thm:OkaUnique}.

Now, suppose that $V(f_{kp})$ is irreducible, and assume that $\pi_1(\PP^2\setminus V(f_{kp}))\cong \ZZ_{kp}$. Then, $F$ induces an epimorphism
$$
F_*:\pi_1(\PP^2\setminus V(f_{kp}))\to \ZZ_p,
$$
where the target is the orbifold fundamental group of $\PP^1$ with respect to the maximal orbifold structure with respect to $F:\PP^2\setminus V(f_{kp})\to\PP^1\setminus\{[0:1]\}$. The kernel of $F_*$ is thus a purely central subgroup which is isomorphic to $\ZZ_k$. The result about $\pi_1\left(\PP^2\setminus (C_B\cup V(f_{kp}))\right)$ follows by induction on the size of $B$ by applying Lemma~\ref{lem:coprime} and Theorems~\ref{thm:additiongeneral} and ~\ref{thm:OkaUnique}. The result about $\pi_1\left(\PP^2\setminus (C_B\cup V(f_{kq}))\right)$ follows the same steps.
\end{proof}

Lastly, we are going to refine the Addition Theorem~\ref{thm:additiongeneral} for the case of plane curve complements whose
fundamental groups are isomorphic to one of Oka's groups from Definition~\ref{def:Gpqr}. Recall that, by
Corollary~\ref{cor:isoGpqr}, every group satisfying Definition~\ref{def:Gpqr} is isomorphic to $G(s'p';s'q';k'q')$ for some
$p',q',s',k'$ such that $\gcd(p',q')=1$ and $p'\geq q'$. Recall also that, by Proposition~\ref{prop:GpqrImpliesFiber},
every plane curve which satisfies that the fundamental group of its complement is isomorphic to one of Oka's groups is a
fiber-type curve.

\begin{thm}[Addition for plane fiber-type curves]\label{thm:additionPlane}
Let $F=[f_{kp}^q:f_{kq}^p]:\PP^2\dashrightarrow\PP^1$ be a pencil satisfying Condition~\ref{cond}. Let $B\subset\PP^1$
be a finite non-empty set such that $C_Q$ is irreducible for all $Q\in B$. Suppose that
$\pi_1(\PP^2\setminus C_B)\cong G(s'p';s'q';k'q')$, where $\gcd(p',q')=1$ and $p'\geq q'$. Assume moreover that one
of the following hold:
\begin{itemize}
\item $\# B=1$ and the kernel of
$F_*:\pi_1(\PP^2\setminus C_B)\twoheadrightarrow \pi_1^{\orb}(\PP^1_{(1,\bar m)})\cong\ZZ_{p'}*\ZZ_{q'}$ is a purely
central subgroup which is finite cyclic, where $\PP^1_{(1,\bar m)}$ is the maximal orbifold structure on $\PP^1$ with
respect to $F:\PP^2\setminus C_B\to \PP^1\setminus B\subset \PP^1$.
\item $\# B\geq 2$.
\end{itemize}
Then, $s'=\# B$, and for all $P\in \PP^1\setminus (B\cup B_F)$,
$$
\pi_1(\PP^2\setminus C_{B\cup\{P\}})\cong G((s'+1)p';(s'+1)q';k'q').
$$
\end{thm}

\begin{proof}
The equality $s'=\# B$ follows from Lemma~\ref{lem:coprime} and Proposition~\ref{prop:center}. If $\# B\geq 2$, Proposition~\ref{prop:GpqrImpliesFiber} implies that $C_B$ is a fiber-type curve corresponding to a pencil $F'$ satisfying Condition~\ref{cond} such that the kernel of
$$
F'_*:\pi_1(\PP^2\setminus C_B)\to \pi_1^{\orb}(\PP^1_{(s',(p',q'))})\cong\FF_{s-1}*\ZZ_{p'}*\ZZ_{q'},
$$
is a purely central subgroup $\pi_1(\PP^2\setminus C_B)$ which is isomorphic to $\ZZ_{k'}$, where $\PP^1$ is endowed with the maximal orbifold structure with respect to $F':\PP^2\setminus C_B\to\PP^1$. Moreover, every irreducible component of $C_B$ is a fiber of $F'$. Hence, the pencils $F'$ and $F$ have two common fibers, and thus they must coincide up to change of coordinates in $\PP^1$. In particular, $F_*$ and $F'_*$ coincide up to isomorphism in the target, and thus have the same kernel.

Suppose that $\# B= 1$. Let $l$ be the order of the kernel of
$F_*:\pi_1(\PP^2\setminus C_B)\twoheadrightarrow \pi_1^{\orb}(\PP^1_{(1,\bar m)})\cong\ZZ_{p'}*\ZZ_{q'}$. Since the kernel of
$F_*$ is purely central, $H_1(\PP^2\setminus C_B)$ is an abelian group of order $lp'q'$. However, since
$\pi_1(\PP^2\setminus C_B)\cong G(p';q';k'q')$, Proposition~\ref{prop:center} implies that $k'p'q'=lp'q'$, so $l=k'$.

We have shown that the kernel of $F_*:\pi_1(\PP^2\setminus C_B)\twoheadrightarrow \pi_1^{\orb}\left(\PP^1_{(s',(p',q'))}\right)$
is a purely central subgroup isomorphic to $\ZZ_{k'}$, where $s'=\# B$ and $\PP^1_{(s',(p',q'))}$ is the maximal orbifold
structure on $\PP^1$ with respect to $F:\PP^2\setminus C_{B}\to\PP^1\setminus B$. The result now follows from Lemma~\ref{lem:coprime}, the Addition Theorem~\ref{thm:additiongeneral} and Theorem~\ref{thm:OkaUnique}.
\end{proof}

\begin{rem}\label{rem:EO}
	Example~\ref{ex:Oka} can be obtained in particular from Theorem~\ref{thm:newgeneric}. However, Eyral and Oka provide similar examples in \cite{Eyral-Oka-RjoinType,Eyral-Oka-RjoinTypeII} for which $C$ is not of generic fiber type. Theorem~\ref{thm:additionPlane}/Theorem~\ref{thm:deletion} can be used to create new curves whose fundamental group of the complement is one of Oka's $G(p;q;r)$ groups, namely adding to $C$ a finite union of generic fibers of the corresponding pencil, or removing some fibers of the pencil from the curve.
\end{rem}

\bibliographystyle{amsplain}

\end{document}